\documentclass{amsart}%
\usepackage{amsmath}
\usepackage{amssymb}
\usepackage{amsfonts}
\usepackage{graphicx}%
\newtheorem{theorem}{Theorem}[section]
\theoremstyle{plain}

\newtheorem{corollary}[theorem]{Corollary}

\newtheorem{definition}[theorem]{Definition}

\newtheorem{lemma}[theorem]{Lemma}

\newtheorem{problem}{Problem}
\newtheorem{proposition}[theorem]{Proposition}
\newtheorem{remark}[theorem]{Remark}

\numberwithin{equation}{section}
\begin{document}
\title[Supercritical problems of Brezis-Nirenberg type]{Ground states of critical and supercritical problems of Brezis-Nirenberg type}
\author{M\'{o}nica Clapp}
\address{Instituto de Matem\'{a}ticas, Universidad Nacional Aut\'{o}noma de M\'{e}xico,
Circuito Exterior, C.U., 04510 M\'{e}xico D.F., Mexico}
\email{monica.clapp@im.unam.mx}
\author{Angela Pistoia}
\address{Dipartimento di Metodi e Modelli Matematici, Universit\'{a} di Roma "La
Sapienza", via Antonio Scarpa 16, 00161 Roma, Italy}
\email{pistoia@dmmm.uniroma1.it}
\author{Andrzej Szulkin}
\address{Department of Mathematics, Stockholm University, 106 91 Stockholm, Sweden}
\email{andrzejs@math.su.se}
\thanks{M. Clapp is partially supported by CONACYT grant 237661 and PAPIIT-DGAPA-UNAM
grant IN104315 (Mexico), {}A. Pistoia is partially supported by PRIN
2009-WRJ3W7 grant (Italy).}
\date{\today}
\maketitle

\begin{abstract}
We study the existence of symmetric ground states to the supercritical problem%
\[
-\Delta v=\lambda v+\left\vert v\right\vert ^{p-2}v\text{ \ in }\Omega,\qquad
v=0\text{ on }\partial\Omega,
\]
in a domain of the form%
\[
\Omega=\{(y,z)\in\mathbb{R}^{k+1}\times\mathbb{R}^{N-k-1}:\left(  \left\vert
y\right\vert ,z\right)  \in\Theta\},
\]
where $\Theta$ is a bounded smooth domain such that $\overline{\Theta}%
\subset\left(  0,\infty\right)  \times\mathbb{R}^{N-k-1},$ $1\leq k\leq N-3,$
$\lambda\in\mathbb{R},$ and $p=\frac{2(N-k)}{N-k-2}$ is the $(k+1)$-st
critical exponent. We show that symmetric ground states exist for $\lambda$ in
some interval to the left of each symmetric eigenvalue, and that no symmetric
ground states exist in some interval $(-\infty,\lambda_{\ast})$ with
$\lambda_{\ast}>0$ if $k\geq2.$

Related to this question is the existence of ground states to the anisotropic
critical problem%
\[
-\text{div}(a(x)\nabla u)=\lambda b(x)u+c(x)\left\vert u\right\vert ^{2^{\ast
}-2}u\quad\text{in}\ \Theta,\qquad u=0\quad\text{on}\ \partial\Theta,
\]
where $a,b,c$ are positive continuous functions on $\overline{\Theta}.$ We
give a minimax characterization for the ground states of this problem, study
the ground state energy level as a function of $\lambda,$ and obtain a
bifurcation result for ground states.\smallskip



\end{abstract}

\section{Introduction}

We consider the supercritical Brezis-Nirenberg type problem%
\begin{equation}
\left\{
\begin{array}
[c]{ll}%
-\Delta v=\lambda v+\left\vert v\right\vert ^{2_{N,k}^{\ast}-2}v & \text{in
}\Omega,\\
\hspace{0.6cm}v=0 & \text{on }\partial\Omega,
\end{array}
\right.  \tag{$\wp_\lambda$}\label{prob}%
\end{equation}
where $\Omega$ is given by%
\begin{equation}
\Omega:=\{(y,z)\in\mathbb{R}^{k+1}\times\mathbb{R}^{N-k-1}:\left(  \left\vert
y\right\vert ,z\right)  \in\Theta\} \label{omega}%
\end{equation}
for some bounded smooth domain $\Theta$ in $\mathbb{R}^{N-k}$ such that
$\overline{\Theta}\subset\left(  0,\infty\right)  \times\mathbb{R}^{N-k-1},$
$1\leq k\leq N-3,$ $\lambda\in\mathbb{R},$ and $2_{N,k}^{\ast}:=\frac
{2(N-k)}{N-k-2}$ is the so-called $(k+1)$-st critical exponent.

If $k=0$ then $2_{N,0}^{\ast}=2^{\ast}$ is the critical Sobolev exponent and
problem (\ref{prob})\ becomes%
\begin{equation}
\left\{
\begin{array}
[c]{ll}%
-\Delta v=\lambda v+\left\vert v\right\vert ^{2^{\ast}-2}v & \text{in }%
\Theta,\\
\hspace{0.6cm}v=0 & \text{on }\partial\Theta.
\end{array}
\right.  \label{bn}%
\end{equation}
A celebrated result by Brezis and Nirenberg \cite{bn} states that (\ref{bn})
has a ground state $v>0$ if and only if $\lambda\in(0,\lambda_{1})$ and
$N\geq4,$ or if $\lambda\in(\lambda_{\ast},\lambda_{1})$ and $N=3,$ where
$\lambda_{\ast}$ is some number in $(0,\lambda_{1}).$ Moreover, they show that
$\lambda_{\ast}=\frac{\lambda_{1}}{4}>0$ if $\Theta$ is a ball. As usual,
$\lambda_{m}$ denotes the $m$-th Dirichlet eigenvalue of $-\Delta$ in
$\Theta.$

Problem (\ref{bn}) has been widely investigated. Capozzi, Fortunato and
Palmieri \cite{cafp} established the existence of solutions for all
$\lambda>0$ if $N\geq5$ and for all $\lambda\ne\lambda_{m}$ if $N=4$ (see also
\cite{gr, w}). Several multiplicity results are also available, see e.g.
\cite{css,cw,ds1,ds2,z}\ and the references therein.

Recently, Szulkin, Weth and Willem \cite{sww} gave a minimax characterization
for the ground states of problem (\ref{bn}) when $\lambda\geq\lambda_{1}$.
They established the existence of ground states for $\lambda\neq\lambda_{m}$
if $N=4$ and for all $\lambda\ge\lambda_{1}$ if $N\geq5.$

Concerning the supercritical problem (\ref{prob}) with $k\geq1,$ Passaseo
\cite{pa1,pa2} showed that a nontrivial solution does not exist if $\lambda=0$
and $\Theta$ is a ball. This statement was extended in \cite{cfp} to more
general domains $\Theta,$ and to some unbounded domains in \cite{cs}. On the
other hand, existence of multiple solutions has been established in \cite{cf,
kp, wy}.

This work is concerned with the existence of symmetric ground states for the
supercritical problem (\ref{prob}) with $k\geq1.$ Note that the domain
$\Omega$ is invariant under the action of the group $O(k+1)$ of linear
isometries of $\mathbb{R}^{k+1}$ on the first $k+1$ coordinates. A function
$v:\Omega\rightarrow\mathbb{R}$ is called $O(k+1)$-invariant if
$v(gy,z)=v(y,z)$ for every $g\in O(k+1),$ $(y,z)\in\mathbb{R}^{k+1}%
\times\mathbb{R}^{N-k-1}.$ The subspace
\[
H_{0}^{1}(\Omega)^{O(k+1)}:=\{v\in H_{0}^{1}(\Omega):v\text{ is }%
O(k+1)\text{-invariant}\}
\]
of $H_{0}^{1}(\Omega)$ is continuously embedded in $L^{2_{N,k}^{\ast}}%
(\Omega),$ so the energy functional $\mathcal{J}_{\lambda}:H_{0}^{1}%
(\Omega)^{O(k+1)}\rightarrow\mathbb{R}$ given by%
\[
\mathcal{J}_{\lambda}(v):=\frac{1}{2}\int_{\Omega}\left\vert \nabla
v\right\vert ^{2}-\frac{\lambda}{2}\int_{\Omega}v^{2}-\frac{1}{2^{\ast}}%
\int_{\Omega}\left\vert v\right\vert ^{2_{N,k}^{\ast}}%
\]
is well defined. Its critical points are the $O(k+1)$-invariant solutions to
problem (\ref{prob}). An $O(k+1)$-invariant $(PS)_{\tau}$-sequence for
$\mathcal{J}_{\lambda}$ is a sequence $(v_{k})$ such that
\[
v_{k}\in H_{0}^{1}(\Omega)^{O(k+1)},\text{\qquad}\mathcal{J}_{\lambda}%
(v_{k})\rightarrow\tau\text{\qquad and\qquad}\mathcal{J}_{\lambda}^{\prime
}(v_{k})\rightarrow0\text{ in }H^{-1}(\Omega).
\]
We set%
\[
\ell_{\lambda}^{O(k+1)}:=\inf\{\tau>0:\text{ there exists an }%
O(k+1)\text{-invariant }(PS)_{\tau}\text{-sequence for }\mathcal{J}_{\lambda
}\}.
\]
This is the lowest possible energy level for a nontrivial $O(k+1)$-invariant
solution to problem (\ref{prob}). An $O(k+1)$\emph{-invariant ground state of
problem} (\ref{prob}) is a critical point $v\in H_{0}^{1}(\Omega)^{O(k+1)}$ of
$\mathcal{J}_{\lambda}$ such that $\mathcal{J}_{\lambda}(v)=\ell_{\lambda
}^{O(k+1)}.$ Since $\mathcal{J}_{\lambda}$ does not satisfy the Palais-Smale
condition, an $O(k+1)$-invariant ground state does not necessarily exist.

Let $0<\lambda_{1}^{\left[  k\right]  }<\lambda_{2}^{\left[  k\right]  }%
\leq\lambda_{3}^{\left[  k\right]  }\leq\cdots$ be the $O(k+1)$-invariant
eigenvalues of the problem%
\[
-\Delta v=\lambda v\quad\text{in}\ \Omega,\qquad v\in H_{0}^{1}(\Omega
)^{O(k+1)},
\]
counted with their multiplicity. Set $\lambda_{0}^{\left[  k\right]  }:=0.$ We
shall prove the following result for $O(k+1)$-invariant ground states.

\begin{theorem}
\label{thm:main}For every $1\leq k\leq N-3,$ the following statements hold true:

\begin{enumerate}
\item[(a)] Problem \emph{(\ref{prob})} does not have an $O(k+1)$-invariant
ground state if $\lambda\leq0.$

\item[(b)] For each $m\in\mathbb{N}\cup\{0\},$ there is a number
$\lambda_{m,\ast}^{\left[  k\right]  }\in\lbrack\lambda_{m}^{\left[  k\right]
},\lambda_{m+1}^{\left[  k\right]  })$ with the property that problem
\emph{(\ref{prob})} has an $O(k+1)$-invariant ground state for every
$\lambda\in(\lambda_{m,\ast}^{\left[  k\right]  },\lambda_{m+1}^{\left[
k\right]  })$ and does not have an $O(k+1)$-invariant ground state for any
$\lambda\in\lbrack\lambda_{m}^{\left[  k\right]  },\lambda_{m,\ast}^{\left[
k\right]  }).$

\item[(c)] Let $\beta:=\max\{$\emph{dist}$(x,\{0\}\times\mathbb{R}%
^{N-k-1}):x\in\overline{\Theta}\}.$ Then,%
\[
\lambda_{0,\ast}^{\left[  k\right]  }\geq\left\{
\begin{array}
[c]{ll}%
\frac{(k-1)^{2}}{4\beta^{2}} & \text{if }3k\geq N,\\
\frac{k}{\left(  2^{\ast}_{N,k}\beta\right)  ^{2}}\left(  (2_{N,k}^{\ast
}-1)k-2_{N,k}^{\ast}\right)  & \text{if }3k\leq N.
\end{array}
\right.
\]
In particular, $\lambda_{0,\ast}^{\left[  k\right]  }>0$ if $k\geq2.$
\end{enumerate}
\end{theorem}

This last statement stands in contrast with the case $k=0$ where a ground
state to problem (\ref{bn}) exists for every $\lambda\in\lbrack0,\lambda_{1})$
if $N\geq4.$ We also show that $\lambda_{0,\ast}^{\left[  1\right]  }>0$ if
$\Theta$ is thin enough, see Proposition \ref{prop:k=1}.

As we shall see, the $O(k+1)$-invariant ground states of problem (\ref{prob})
correspond to the ground states of the critical problem
\begin{equation}
-\text{div}(a(x)\nabla u)=\lambda b(x)u+c(x)\left\vert u\right\vert ^{2^{\ast
}-2}u\quad\text{in}\ \Theta,\qquad u=0\quad\text{on}\ \partial\Theta,
\label{prob1}%
\end{equation}
with\ $2^{\ast}=\frac{2n}{n-2},$ $n:=\dim\Theta,$ $a(x_{1},...,x_{n}%
)=x_{1}^{k}$ and $a=b=c.$

The critical problem (\ref{prob1}) with general coefficients $a\in
\mathcal{C}^{1}(\overline{\Theta}),$ $b,c\in\mathcal{C}^{0}(\overline{\Theta
})$ has an interest in its own. We study it in section \ref{sec:anisotropic}
and give a minimax characterization for its ground states, similar to that in
\cite{sww}. We study the properties of its ground state energy level as a
function of $\lambda,$ and obtain a bifurcation result for ground states, see
Theorem \ref{thm:bifurcation}.

Anisotropic critical problems of the form (\ref{prob1}) have been studied, for
example, by Egnell \cite{e} and, more recently, by Hadiji et al. \cite{hmpy,
hy}. They obtained existence and multiplicity results under some assumptions
which involve flatness of the coefficient functions at some local maximum or
minimum point in the interior of $\Theta.$ Note that the function
$a(x_{1},...,x_{n})=x_{1}^{k}$ attains its minimum on the boundary of
$\Theta.$ This produces a quite different behavior regarding the existence of
ground states, as we shall see in the following sections.

Section \ref{sec:anisotropic} is devoted to the study of the general
anisotropic critical problem. In section \ref{sec:nonexistence} we prove a
nonexistence result for supercritical problems. It will be used in Section
\ref{sec:proof} where we prove Theorem \ref{thm:main}. In the last section we
include some questions and remarks.

\section{Ground states of the anisotropic critical problem}

\label{sec:anisotropic}In this section we consider the anisotropic
Brezis-Nirenberg type problem
\begin{equation}
\left\{
\begin{array}
[c]{ll}%
-\,\text{div}(a(x)\nabla u)=\lambda b(x)u+c(x)\left\vert u\right\vert
^{{2}^{\ast}-2}u & \text{in}\ \Theta,\\
\hspace{2.2cm}u=0 & \text{on}\ \partial\Theta,
\end{array}
\right.  \label{ancrit}%
\end{equation}
where $\Theta$ is a bounded smooth domain in $\mathbb{R}^{n},$ $n\geq3,$
$\lambda\in\mathbb{R}$, $a\in\mathcal{C}^{1}(\overline{\Theta}),$
$b,c\in\mathcal{C}^{0}(\overline{\Theta})$ are strictly positive on
$\overline{\Theta},$ and $2^{\ast}:=\frac{2n}{n-2}$ is the critical Sobolev
exponent in dimension $n.$

We take
\begin{equation}
\left\langle u,v\right\rangle _{a}:=\int_{\Theta}a(x)\nabla u\cdot\nabla
v,\text{\qquad}\left\Vert u\right\Vert _{a}:=\left(  \int_{\Theta
}a(x)\left\vert \nabla u\right\vert ^{2}\right)  ^{1/2}, \label{scalprod}%
\end{equation}
to be the scalar product and the norm in $H_{0}^{1}(\Theta),$ and%
\[
\left\vert u\right\vert _{b,2}:=\left(  \int_{\Theta}b(x)u^{2}\right)
^{1/2},\qquad\left\vert u\right\vert _{c,2^{\ast}}:=\left(  \int_{\Theta
}c(x)\left\vert u\right\vert ^{2^{\ast}}\right)  ^{1/2^{\ast}},
\]
to be the norms in $L^{2}(\Theta)$ and $L^{2^{\ast}}(\Theta)$ respectively.
They are, clearly, equivalent to the standard ones.

Let $0<\lambda_{1}^{a,b}<\lambda_{2}^{a,b}\leq\lambda_{3}^{a,b}\leq\cdots$ be
the eigenvalues of the problem%
\[
-\text{div}(a(x)\nabla u)=\lambda b(x)u\quad\text{in}\ \Theta,\qquad
u=0\quad\text{on}\ \partial\Theta,
\]
counted with their multiplicity, and $e_{1},e_{2},e_{3},...$ be the
corresponding normalized eigenfunctions, i.e. $\left\vert e_{j}\right\vert
_{b,2}=1.$ Set
\begin{align*}
Z_{0}  &  :=\{0\},\qquad Z_{m}:=\text{span}\{e_{1},...,e_{m}\},\\
Y_{m}  &  :=\{w\in H_{0}^{1}(\Theta):\left\langle w,z\right\rangle
_{a}=0\text{ for all }z\in Z_{m}\},\\
T_{0}  &  :=(-\infty,\lambda_{1}^{a,b}), \quad\text{and} \quad T_{m}%
:=[\lambda_{m}^{a,b},\lambda_{m+1}^{a,b}) \text{ if } m\in\mathbb{N}.
\end{align*}

The solutions to problem (\ref{ancrit}) are the critical points of the
functional $J_{\lambda}:H_{0}^{1}(\Theta)\rightarrow\mathbb{R}$ given by%
\[
J_{\lambda}(u):=\frac{1}{2}\left\Vert u\right\Vert _{a}^{2}-\frac{\lambda}%
{2}\left\vert u\right\vert _{b,2}^{2}-\frac{1}{2^{\ast}}\left\vert
u\right\vert _{c,2^{\ast}}^{2^{\ast}}.
\]
If $\lambda\in T_{m}$ we define%
\[
\mathcal{N}_{\lambda}\equiv\mathcal{N}_{\lambda}(\Theta):=\{u\in H_{0}%
^{1}(\Theta)\smallsetminus Z_{m}:J_{\lambda}^{\prime}(u)u=0\text{ and
}J_{\lambda}^{\prime}(u)z=0\text{ for all }z\in Z_{m}\}.
\]
This is a $\mathcal{C}^{1}$-submanifold of codimension $m+1$ in $H_{0}%
^{1}(\Theta)$, cf. \cite{sww}. If $\lambda<\lambda_{1}^{a,b}$ it is the usual
Nehari manifold, and if $\lambda\geq\lambda_{1}^{a,b}$ it is the generalized
Nehari manifold, introduced by Pankov in \cite{p} and studied by Szulkin and
Weth in \cite{sw1, sw2}. Note that $J_{\lambda}^{\prime}(z)z<0$ for all $z\in
Z_{m}\smallsetminus\{0\}.$ Clearly, the nontrivial critical points of
$J_{\lambda}$ belong to $\mathcal{N}_{\lambda}.$ Moreover, they coincide with
the critical points of its restriction $J_{\lambda}|_{\mathcal{N}_{\lambda}%
}:\mathcal{N}_{\lambda}\rightarrow\mathbb{R}$. The proof of these facts is
completely analogous to the one given in \cite{sww}\ for the autonomous case.
Set%
\[
\ell_{\lambda}\equiv\ell_{\lambda}^{a,b,c}:=\inf_{\mathcal{N}_{\lambda}%
}J_{\lambda}.
\]
Following \cite{sw1} one shows that, for every $w\in Y_{m}\smallsetminus
\{0\},$ there exist unique $t_{\lambda,w}\in(0,\infty)$ and $z_{\lambda,w}\in
Z_{m}$ such that%
\[
t_{\lambda,w}w+z_{\lambda,w}\in\mathcal{N}_{\lambda},
\]
and that%
\[
J_{\lambda}(t_{\lambda,w}w+z_{\lambda,w})=\max_{t>0,\,z\in Z_{m}}J_{\lambda
}(tw+z).
\]
Let $\Sigma_{m}:=\{w\in Y_{m}:\left\Vert w\right\Vert _{a}=1\}$ be the unit
sphere in $Y_{m}.$ Then,%
\begin{equation}
\label{aaa}\ell_{\lambda}=\inf_{w\in\Sigma_{m}}\max_{\substack{t>0,\\z\in
Z_{m}}}J_{\lambda}(tw+z).
\end{equation}

As usual, we denote the best Sobolev constant for the embedding $H^{1}%
(\mathbb{R}^{n})\hookrightarrow L^{2^{\ast}}(\mathbb{R}^{n})$ by $S.$ We set%
\[
\kappa^{a,c}:=\left(  \min\limits_{x\in\overline{\Theta}}\frac{a(x)^{\frac
{n}{2}}}{c(x)^{\frac{n-2}{2}}}\right)  \frac{1}{n}S^{\frac{n}{2}},
\]
and define%
\[
\lambda_{m,\ast}^{a,b,c}:=\inf\{\lambda\in T_{m}:\ell_{\lambda}<\kappa
^{a,c}\}.
\]

\begin{theorem}
\label{thm:bifurcation}For every $m\in\mathbb{N}\cup\{0\}$ the following
statements hold true:

\begin{enumerate}
\item[(a)] The function $\lambda\longmapsto\ell_{\lambda}$ is nonincreasing in
$T_{m}$ and
\[
0<\ell_{\lambda}\leq\kappa^{a,c}\text{\qquad for all \ }\lambda\in T_{m}.
\]

\item[(b)] $\ell_{\lambda}$ is attained on $\mathcal{N}_{\lambda}$ if
$\ell_{\lambda}<\kappa^{a,c}.$

\item[(c)] The function $\lambda\longmapsto\ell_{\lambda}$ is continuous in
$T_{m}$ and
\[
\lim_{\lambda\nearrow\lambda_{m+1}^{a,b}}\ell_{\lambda}=0.
\]
Hence, $\lambda_{m,\ast}^{a,b,c}<\lambda_{m+1}^{a,b}.$

\item[(d)] $\ell_{\lambda}$ is not attained if $\lambda\in(-\infty
,\lambda_{0,\ast}^{a,b,c})$ or $\lambda\in\lbrack\lambda_{m}^{a,b}%
,\lambda_{m,\ast}^{a,b,c}),$ $m\geq1,$ and is attained if $\lambda\in
(\lambda_{m,\ast}^{a,b,c},\lambda_{m+1}^{a,b}).$
\end{enumerate}
\end{theorem}

\begin{remark}
\emph{ It follows from part (c) above that bifurcation (to the left) occurs at
each $\lambda_{m}^{a,b}$. This fact is essentially known and can be obtained
by other methods. However, we would like to emphasize that here we show that
our bifurcating solutions are\emph{ ground states}. }
\end{remark}

\begin{proof}
[Proof of Theorem 2.1](a):\quad Let $\lambda,\mu\in T_{m}.$ If $\lambda\leq
\mu$ then $J_{\lambda}(u)\geq J_{\mu}(u)$ for every $u\in H_{0}^{1}(\Theta).$
So $\ell_{\lambda}\geq\ell_{\mu}$ according to \eqref{aaa}. This proves that
$\lambda\longmapsto\ell_{\lambda}$ is nonincreasing in $T_{m}$.

If $\lambda\in T_{m}$ and $w\in\Sigma_{m}$ we have that%
\begin{align}
\max_{t>0,\,z\in Z_{m}}J_{\lambda}(tw+z)  &  \geq\max_{t>0}J_{\lambda
}(tw)=\frac{1}{n}\left(  \frac{\left\Vert w\right\Vert _{a}^{2}-\lambda
\left\vert w\right\vert _{b,2}^{2}}{\left\vert w\right\vert _{c,2^{\ast}}^{2}%
}\right)  ^{n/2}\label{compNeh}\\
&  \geq\frac{1}{n}\left(  \frac{1-\frac{\lambda}{\lambda_{m+1}}}{\left\vert
w\right\vert _{c,2^{\ast}}^{2}}\right)  ^{n/2}.\nonumber
\end{align}
Using Sobolev's inequality we conclude that there is a positive constant $C$
such that%
\[
\max_{t>0,\,z\in Z_{m}}J_{\lambda}(tw+z)\geq C\text{\qquad for all }w\in
\Sigma_{m}.
\]
Therefore, $\ell_{\lambda}>0.$

Let $\varphi_{k}\in\mathcal{C}_{c}^{\infty}(\mathbb{R}^{n})$ be a positive
function such that supp$(\varphi_{k})\subset B_{1/k}(0)$ and $\int\left\vert
\nabla\varphi_{k}\right\vert ^{2}\rightarrow S^{n/2}$, $\int\left\vert
\varphi_{k}\right\vert ^{2^{\ast}}\rightarrow S^{n/2},$ where $B_{r}%
(\xi)=\{x\in\mathbb{R}^{n}:\left\vert x-\xi\right\vert <r\}.$ Let $\xi
\in\overline{\Theta}$ be such that
\[
\frac{a(\xi)^{\frac{n}{2}}}{c(\xi)^{\frac{n-2}{2}}}=\min\limits_{x\in
\overline{\Theta}}\frac{a(x)^{\frac{n}{2}}}{c(x)^{\frac{n-2}{2}}}%
\]
and choose $\nu\in\mathbb{R}^{n}$ with $\left\vert \nu\right\vert =1$ such
that $\nu$ is the inward pointing unit normal at $\xi$ if $\xi\in
\partial\Theta.$ Set $\xi_{k}:=\xi+\frac{1}{k}\nu$ and $u_{k}(x):=\varphi
_{k}(x-\xi_{k}).$ Then $u_{k}\in H_{0}^{1}(\Theta)$ for $k$ large enough, and
we have that%
\begin{align}
\max_{t>0}J_{\lambda}(tu_{k})  &  =\frac{1}{n}\left(  \frac{\left\Vert
u_{k}\right\Vert _{a}^{2}-\lambda\left\vert u_{k}\right\vert _{b,2}^{2}%
}{\left\vert u_{k}\right\vert _{c,2^{\ast}}^{2}}\right)  ^{\frac{n}{2}%
}\label{eq:1}\\
&  =\frac{1}{n}\left(  \frac{\int_{B_{1/k}(\xi_{k})}a(x)\left\vert \nabla
u_{k}\right\vert ^{2}-\lambda\int_{B_{1/k}(\xi_{k})}b(x)u_{k}^{2}}{\left(
\int_{B_{1/k}(\xi_{k})}c(x)\left\vert u_{k}\right\vert ^{2^{\ast}}\right)
^{2/2^{\ast}}}\right)  ^{\frac{n}{2}}\nonumber\\
&  \longrightarrow\frac{1}{n}\left(  \frac{a(\xi)^{\frac{n}{2}}}{c(\xi
)^{\frac{n-2}{2}}}\right)  S^{\frac{n}{2}}=\kappa^{a,c}\qquad\text{as
}k\rightarrow\infty.\nonumber
\end{align}
Hence, $\ell_{\lambda}\leq\kappa^{a,c}$ for $\lambda<\lambda_{1}^{a,b}.$

Next, we assume that $\lambda\in T_{m}$ with $m\in\mathbb{N}$. We fix an open
subset $\theta$ of $\Theta$ such that $\theta\cap B_{1/k}(\xi_{k})=\emptyset$
for $k$ large enough. If $z\in Z_{m}$ and $z=0$ in $\theta$ then $z=0$ in
$\Theta,$ see \cite[Lemma 3.3]{sww}. Hence, $(\int_{\theta}c(x)\left\vert
z\right\vert ^{2^{\ast}})^{1/2^{\ast}}$ is a norm in $Z_{m}$ and, since
$Z_{m}$ is finite-dimensional, this norm is equivalent to $\left\Vert
z\right\Vert _{a}.$ In particular, there is a positive constant $A$ such that
$\int_{\theta}c(x)\left\vert z\right\vert ^{2^{\ast}}\geq2^{\ast}A\left\Vert
z\right\Vert _{a}^{2^{\ast}}$ for all $z\in Z_{m}.$ It follows by convexity
that, for every $t>0$ and every $z\in Z_{m},$ we have%
\begin{align*}
\left\vert tu_{k}+z\right\vert _{c,2^{\ast}}^{2^{\ast}}  &  =\int
_{\Theta\smallsetminus\theta}c(x)\left\vert tu_{k}+z\right\vert ^{2^{\ast}%
}+\int_{\theta}c(x)\left\vert z\right\vert ^{2^{\ast}}\\
&  \geq t^{2^{\ast}}\int_{\Theta}c(x)u_{k}^{2^{\ast}}+2^{\ast}t^{2^{\ast}%
-1}\int_{\Theta}c(x)u_{k}^{2^{\ast}-1}z+2^{\ast}A\left\Vert z\right\Vert
_{a}^{2^{\ast}}.
\end{align*}
Therefore,%
\begin{align}
J_{\lambda}(tu_{k}+z)  &  \leq J_{0}(tu_{k})-\frac{\lambda}{2}\left\vert
tu_{k}\right\vert _{b,2}^{2}+t\int_{\Theta}\left(  a(x)\nabla u_{k}\nabla
z-\lambda b(x)u_{k}z\right) \label{eq:2}\\
&  \qquad+\frac{1}{2}\left(  \left\Vert z\right\Vert _{a}^{2}-\lambda
\left\vert z\right\vert _{b,2}^{2}\right)  -t^{2^{\ast}-1}\int_{\Theta
}c(x)u_{k}^{2^{\ast}-1}z-A\left\Vert z\right\Vert _{a}^{2^{\ast}}\nonumber\\
&  \leq J_{0}(tu_{k})+t\int_{\Theta}\left(  a(x)\nabla u_{k}\nabla z-\lambda
b(x)u_{k}z\right) \nonumber\\
&  \qquad-t^{2^{\ast}-1}\int_{\Theta}c(x)u_{k}^{2^{\ast}-1}z-A\left\Vert
z\right\Vert _{a}^{2^{\ast}}.\nonumber
\end{align}
Consequently,%
\[
J_{\lambda}(tu_{k}+z)\leq B(t^{2}+t\left\Vert z\right\Vert _{a}+t^{2^{\ast}%
-1}\left\Vert z\right\Vert _{a})-C(t^{2^{\ast}}+\left\Vert z\right\Vert
_{a}^{2^{\ast}})
\]
for some positive constants $B$ and $C$. This implies that there exists $R>0$
such that $J_{\lambda}(tu_{k}+z)\leq0$ for all $t\geq R,$ $z\in Z_{m}$ and $k$
large enough. On the other hand, for $t\leq R,$ $z\in Z_{m}$ and $k$ large
enough, since $\varphi_{k}\rightharpoonup0$ weakly in $H_{0}^{1}(\Theta),$
inequalities (\ref{eq:2}) and (\ref{eq:1}) imply that%
\[
J_{\lambda}(tu_{k}+z)\leq J_{0}(tu_{k})+o(1)=\kappa^{a,c}+o(1).
\]
This proves that $\ell_{\lambda}\leq\kappa^{a,c}$ for $\lambda\geq\lambda
_{1}^{a,b}$ and concludes the proof of statement (a).

(b):\quad Let $I_{\lambda}:\Sigma_{m}\rightarrow\mathbb{R}$ be the function
given by
\[
I_{\lambda}(w):=J_{\lambda}(t_{\lambda,w}w+z_{\lambda,w}).
\]
Then $\ell_{\lambda}:=\inf_{w\in\Sigma_{m}}I_{\lambda}(w).$ It is shown in
\cite{sw1, sw2}\ that $I_{\lambda}\in\mathcal{C}^{1}(\Sigma_{m},\mathbb{R})$.
Since $\Sigma_{m}$ is a smooth submanifold of $H_{0}^{1}(\Theta)$, Ekeland's
variational principle yields a Palais-Smale sequence $(w_{k})$ for
$I_{\lambda}$ such that $I_{\lambda}(w_{k})\rightarrow\ell_{\lambda},$ cf.
\cite[Theorem 8.5]{w}. Set $u_{k}:=t_{\lambda,w_{k}}w_{k}+z_{\lambda,w_{k}}.$
By Corollary 2.10 in \cite{sw1} or Corollary 33 in \cite{sw2}, $(u_{k})$ is a
Palais-Smale sequence for $J_{\lambda}.$ Now, Corollary 3.2 in \cite{cf}
asserts that every Palais-Smale sequence $(u_{k})$ for $J_{\lambda}$ such that
$J_{\lambda}(u_{k})\rightarrow\tau<\kappa^{a,c},$ contains a convergent
subsequence. It follows that $\ell_{\lambda}$ is attained on $\mathcal{N}%
_{\lambda}$ if $\ell_{\lambda}<\kappa^{a,c}.$

(c):\quad Let $w\in\Sigma_{m}$. First, we will show that the function
$\lambda\longmapsto I_{\lambda}(w)$ is continuous in $T_{m}.$ Let $\mu_{j}%
,\mu\in T_{m}$ be such that $\mu_{j}\rightarrow\mu.$ A standard argument shows
that $J_{\mu_{j}}(tw+z)\leq0$ for every $j\in\mathbb{N}$ if $t^{2}+\left\Vert
z\right\Vert _{a}^{2}$ is large enough. Therefore, the sequences $(t_{\mu
_{j},w})$ and $(z_{\mu_{j},w})$ are bounded and, after passing to a
subsequence, $t_{\mu_{j},w}\rightarrow t_{0}$ in $[0,\infty)$ and $z_{\mu
_{j},w}\rightarrow z_{0}$ in $Z_{m}.$ Hence,%
\[
I_{\mu_{j}}(w)=J_{\mu_{j}}(t_{\mu_{j},w}w+z_{\mu_{j},w})\rightarrow J_{\mu
}(t_{0}w+z_{0})\leq I_{\mu}(w).
\]
If $J_{\mu}(t_{0}w+z_{0})<I_{\mu}(w)$ then, since%
\[
J_{\mu_{j}}(t_{\mu,w}w+z_{\mu,w})\rightarrow J_{\mu}(t_{\mu,w}w+z_{\mu
,w})=I_{\mu}(w),
\]
we would have that, for $j$ large enough,%
\[
\max_{t>0,\,z\in Z_{m}}J_{\mu_{j}}(tw+z)=J_{\mu_{j}}(t_{\mu_{j},w}w+z_{\mu
_{j},w})<J_{\mu_{j}}(t_{\mu,w}w+z_{\mu,w})\text{,}%
\]
which is a contradiction. Consequently, $I_{\mu_{j}}(w)\rightarrow I_{\mu
}(w).$ This proves that $\lambda\longmapsto I_{\lambda}(w)$ is continuous in
$T_{m}$ for each $w\in\Sigma_{m}.$

Next, we prove that the function $\lambda\longmapsto\ell_{\lambda}$ is
continuous from the left in $T_{m}$. Let $\mu_{j},\mu\in T_{m}$ be such that
$\mu_{j}\leq\mu$ and $\mu_{j}\rightarrow\mu.$ Since the infimum of any family
of continuous functions is upper semicontinuous and $\lambda\longmapsto
\ell_{\lambda}$ is nonincreasing, we have that
\[
\limsup_{j\rightarrow\infty}\ell_{\mu_{j}}\leq\ell_{\mu}\leq\liminf
_{j\rightarrow\infty}\ell_{\mu_{j}}.
\]
This proves that $\lambda\longmapsto\ell_{\lambda}$ is continuous from the
left in $T_{m}$.

To prove that $\lambda\longmapsto\ell_{\lambda}$ is continuous from the right
in $T_{m}$ we argue by contradiction. Assume there are $\mu_{j},\mu\in T_{m}$
such that $\mu_{j}\geq\mu$, $\mu_{j}\rightarrow\mu$ and $\sup_{j\in\mathbb{N}%
}\ell_{\mu_{j}}<\ell_{\mu}.$ Then $\ell_{\mu_{j}}<\kappa^{a,c}$ and, by
statement (b), there exists $w_{j}\in\Sigma_{m}$ such that $\ell_{\mu_{j}%
}=J_{\mu_{j}}(t_{\mu_{j},w}w_{j}+z_{\mu_{j},w}).$ Inequality (\ref{compNeh})
asserts that
\[
\ell_{\mu}>\ell_{\mu_{j}}=J_{\mu_{j}}(t_{\mu_{j},w_{j}}w_{j}+z_{\mu_{j},w_{j}%
})\geq\frac{1}{n}\left(  \frac{1-\frac{\mu_{j}}{\lambda_{m+1}}}{\left\vert
w_{j}\right\vert _{c,2^{\ast}}^{2}}\right)  ^{n/2}.
\]
This implies that\ $\left\vert w_{j}\right\vert _{c,2^{\ast}}^{2^{\ast}}%
\geq\varepsilon>0$ for all $j\in\mathbb{N}.$ Denote the closure of $Y_{m}$ in
$L^{2^{\ast}}(\Theta)$ by $\widetilde{Y}_{m}$. Since dim$(Z_{m})<\infty,$ the
projection $\widetilde{Y}_{m}\oplus Z_{m}\rightarrow\widetilde{Y}_{m}$ is
continuous in $L^{2^{\ast}}(\Theta).$ Hence, there is a positive constant
$A_{0}$ such that%
\begin{align*}
\ell_{\mu}  &  \leq J_{\mu}(t_{\mu,w_{j}}w_{j}+z_{\mu,w_{j}})\\
&  =\frac{t_{\mu,w_{j}}^{2}}{2}(1-\mu\left\vert w_{j}\right\vert _{b,2}%
^{2})+\frac{1}{2}(\left\Vert z_{\mu,w_{j}}\right\Vert _{a}^{2}-\mu\left\vert
z_{\mu,w_{j}}\right\vert _{b,2}^{2})-\frac{1}{2^{\ast}}\left\vert t_{\mu
,w_{j}}w_{j}+z_{\mu,w_{j}}\right\vert _{c,2^{\ast}}^{2^{\ast}}\\
&  \leq\frac{t_{\mu,w_{j}}^{2}}{2}-A_{0}\frac{t_{\mu,w_{j}}^{2^{\ast}}%
}{2^{\ast}}\left\vert w_{j}\right\vert _{c,2^{\ast}}^{2^{\ast}}\leq
\frac{t_{\mu,w_{j}}^{2}}{2}-A_{0}\varepsilon\frac{t_{\mu,w_{j}}^{2^{\ast}}%
}{2^{\ast}}\text{\qquad for all }j\in\mathbb{N}.
\end{align*}
It follows that $(t_{\mu,w_{j}})$ is bounded. Hence, $(\Vert z_{\mu,w_{j}%
}\Vert_{a})$ is bounded too. Consequently,%
\begin{align*}
\ell_{\mu}  &  \leq J_{\mu}(t_{\mu,w_{j}}w_{j}+z_{\mu,w_{j}})=J_{\mu_{j}%
}(t_{\mu,w_{j}}w_{j}+z_{\mu,w_{j}})+(\mu-\mu_{j})\left\vert t_{\mu,w_{j}}%
w_{j}+z_{\mu,w_{j}}\right\vert _{b,2}^{2}\\
&  \leq J_{\mu_{j}}(t_{\mu_{j},w_{j}}w_{j}+z_{\mu_{j},w_{j}})+o(1)=\ell
_{\mu_{j}}+o(1)\leq\sup_{j\in\mathbb{N}}\ell_{\mu_{j}}+o(1)<\ell_{\mu}+o(1).
\end{align*}
This is a contradiction. It follows that the function $\lambda\longmapsto
\ell_{\lambda}$ is continuous in $T_{m}$.

Finally, let $\mu_{j}\in T_{m}$ be such that $\mu_{j}\rightarrow\lambda
_{m+1}.$ We have that%
\begin{align*}
0  &  <\ell_{\mu_{j}}\leq J_{\mu_{j}}(t_{\mu_{j},e_{m+1}}e_{m+1}+z_{\mu
_{j},e_{m+1}})\\
&  =\frac{t_{\mu_{j},e_{m+1}}^{2}}{2}(\lambda_{m+1}-\mu_{j})+\frac{1}%
{2}(\left\Vert z_{\mu_{j},e_{m+1}}\right\Vert _{a}^{2}-\mu_{j}\left\vert
z_{\mu_{j},e_{m+1}}\right\vert _{b,2}^{2})\\
&  \qquad-\frac{1}{2^{\ast}}\left\vert t_{\mu_{j},e_{m+1}}e_{m+1}+z_{\mu
_{j},e_{m+1}}\right\vert _{c,2^{\ast}}^{2^{\ast}}\\
&  \leq\frac{t_{\mu_{j},e_{m+1}}^{2}}{2}(\lambda_{m+1}-\mu_{j})-A_{0}%
\frac{t_{\mu_{j},e_{m+1}}^{2^{\ast}}}{2^{\ast}}\left\vert e_{m+1}\right\vert
_{c,2^{\ast}}^{2^{\ast}}.
\end{align*}
It follows that $(t_{\mu_{j},e_{m+1}})$ is bounded and, hence, that
\[
0<\ell_{\mu_{j}}\leq\frac{t_{\mu_{j},e_{m+1}}^{2}}{2}(\lambda_{m+1}-\mu
_{j})=o(1).
\]
This proves that $\ell_{\mu_{j}}\rightarrow0$ as\ $\mu_{j}\rightarrow
\lambda_{m+1}$ from the left.

(d):$\quad$If $\lambda\in T_{m},$ $\lambda\leq\lambda_{m,\ast}^{a,b,c},$ and
$w\in\Sigma_{m}$ were such that $\ell_{\lambda}=I_{\lambda}(w)$ then for
$\mu\in(\lambda,\lambda_{m,\ast}^{a,b,c}(\Theta))$ we would have that
\[
\kappa^{a,c}=\ell_{\mu}\leq I_{\mu}(w)<I_{\lambda}(w)=\ell_{\lambda},
\]
contradicting (a). It follows that $\ell_{\lambda}$ is not attained if
$\lambda\in\lbrack\lambda_{m}^{a,b},\lambda_{m,\ast}^{a,b,c}).$ Statement (b)
implies that $\ell_{\lambda}$ is attained if $\lambda\in(\lambda_{m,\ast
}^{a,b,c},\lambda_{m+1}^{a,b}).$
\end{proof}

Recall that a $(PS)_{\tau}$-sequence for $J_{\lambda}$ is a sequence $(u_{k})$
in $H_{0}^{1}(\Theta)$ such that $J_{\lambda}(u_{k})\rightarrow\tau$ and
$J_{\lambda}^{\prime}(u_{k})\rightarrow0$ in $H^{-1}(\Theta).$ The value
$\ell_{\lambda}$ is characterized as follows.

\begin{corollary}
\label{cor:ps}$\ell_{\lambda}=\inf\{\tau>0:$ there exists a $(PS)_{\tau}%
$-sequence for $J_{\lambda}\}.$
\end{corollary}

\begin{proof}
The argument given in the proof of statement (b) of Theorem
\ref{thm:bifurcation} shows that there exists a $(PS)_{\ell_{\lambda}}%
$-sequence for $J_{\lambda}.$ To prove that $\ell_{\lambda}$ is the smallest
positive number with this property, we argue by contradiction. Assume that
$\tau<\ell_{\lambda}$ and that there exists a $(PS)_{\tau}$-sequence for
$J_{\lambda}$. Then $\tau<\kappa^{a,c}$ and Corollary 3.2 in \cite{cf} asserts
that $(u_{k})$ contains a subsequence which converges to a critical point $u$
of $J_{\lambda}$ with $J_{\lambda}(u)=\tau$. If $\tau\neq0$ then
$u\in\mathcal{N}_{\lambda}$ and, hence, $\ell_{\lambda}\leq\tau$. This is a contradiction.
\end{proof}

For the classical Brezis-Nirenberg problem (\ref{bn}) (where $a=b=c\equiv1$)
with $n\geq4,$ it is known that $\lambda_{0,\ast}^{a,b,c}=0$ and
$\lambda_{m,\ast}^{a,b,c}=\lambda_{m},$ the $m$-th Dirichlet eigenvalue of
$-\Delta$ in $\Theta,$ for all $m\in\mathbb{N}$. Moreover, $\ell_{\lambda
}=\frac{1}{n}S^{\frac{n}{2}}=\kappa^{a,c}$ for every $\lambda\leq0,$ but
$\ell_{\lambda}<\frac{1}{n}S^{\frac{n}{2}}$ for every $\lambda>0$ if $n\geq5,$
see \cite{bn, gr, sww}.

As we shall see below, this is not true in general: For the problem
(\ref{prob2}) in Section \ref{sec:proof} which arises from the supercritical
one, one has that $\lambda_{0,\ast}^{a,b,c}>0$ in most cases, see Propositions
\ref{prop:k>1} and \ref{prop:k=1}. A special feature of that problem is that
the value $\kappa^{a,c}$ is attained on the boundary of $\Theta.$ A different
situation was considered by Egnell \cite{e} and Hadiji and Yazidi \cite{hy}.
They showed for example that, if $a$ attains its minimum at an interior point
$x_{0}$ of $\Theta,$ $b=1=c,$ and $a$ is flat enough around $x_{0},$ then
$\lambda_{0,\ast}^{a,b,c}=0$ for $n\geq4$, as in the classical
Brezis-Nirenberg case.

We do not know whether, in general, $\lambda_{0,\ast}^{a,b,c}\geq0$. But this
will be true in the special case we are interested in, see Proposition
\ref{prop:nogroundstate}. The proof uses a nonexistence result for the
supercritical problem, which we discuss in the following section.

\section{Nonexistence of solutions to a supercritical problem}

\label{sec:nonexistence}Let $\Theta$ be a bounded smooth domain in
$\mathbb{R}^{N-k}$ with $\overline{\Theta}\subset\left(  0,\infty\right)
\times\mathbb{R}^{N-k-1}$ and $0\leq k\leq N-3$. Set
\[
\Omega:=\{(y,z)\in\mathbb{R}^{k+1}\times\mathbb{R}^{N-k-1}:\left(  \left\vert
y\right\vert ,z\right)  \in\Theta\}
\]
and consider the problem
\begin{equation}
\left\{
\begin{array}
[c]{ll}%
-\Delta u=\lambda u+\left\vert u\right\vert ^{p-2}u & \text{in }\Omega,\\
\hspace{0.6cm}u=0 & \text{on }\partial\Omega.
\end{array}
\right.  \label{prob_p}%
\end{equation}

Passaseo \cite{pa1,pa2} showed that, if $\Theta$ is a ball, problem
(\ref{prob_p}) does not have a nontrivial solution for $\lambda=0$ and
$p\geq2_{N,k}^{\ast}:=\frac{2(N-k)}{N-k-2}$. In \cite{cfp} it is shown\ that
this is also true for doubly starshaped domains.

\begin{definition}
\label{def:**}$\Theta$ is doubly starshaped if there exist two numbers
$0<t_{0}<t_{1}$ such that $t\in(t_{0},t_{1})$ for every $(t,z)\in$ $\Theta$
and $\Theta$ is strictly starshaped with respect to $\xi_{0}:=(t_{0},0)$ and
to $\xi_{1}:=(t_{1},0)$, i.e.
\[
\left\langle x-\xi_{i},\nu_{\Theta}(x)\right\rangle >0\qquad\forall
x\in\partial\Theta\smallsetminus\left\{  \xi_{i}\right\}  ,\text{ \ }i=0,1,
\]
where $\nu_{\Theta}$ is the outward pointing unit normal to $\partial\Theta.$
\end{definition}

We denote the first Dirichlet eigenvalue of $-\Delta$ in $\Omega$ by
$\lambda_{1}(\Omega).$

\begin{theorem}
\label{thm:nonexistence}If $\Theta$ is doubly starshaped, $p\geq2_{N,k}^{\ast
}$ and%
\[
\lambda\leq\frac{2(p-2_{N,k}^{\ast})}{2_{N,k}^{\ast}(p-2)}\lambda_{1}%
(\Omega),
\]
then problem \emph{(\ref{prob_p})} does not have a nontrivial solution.
\end{theorem}

We point out that the geometric assumption on $\Theta$ cannot be dropped.
Existence of multiple solutions to problem (\ref{prob_p}) for $\lambda=0$ and
$p=2_{N,k}^{\ast}$ in some domains where $\Theta$ is not doubly starshaped has
been established in \cite{cf,kp, wy}.

The proof of Theorem \ref{thm:nonexistence} follows the ideas introduced in
\cite{cfp,pa1,pa2}. Fix $\tau\in(0,\infty)$ and let $\varphi$ be the solution
to the problem%
\[
\left\{
\begin{array}
[c]{ll}%
\varphi^{\prime}(t)t+(k+1)\varphi(t)=1, & t\in(0,\infty),\\
\varphi(\tau)=0. &
\end{array}
\right.
\]
Explicitly, $\varphi(t)=\frac{1}{k+1}\left[  1-(\frac{\tau}{t})^{k+1}\right]
.$ Note that $\varphi$ is strictly increasing in $(0,\infty).$ For $y\neq0$ we
define%
\begin{equation}
\chi_{\tau}(y,z):=(\varphi(\left\vert y\right\vert )y,z). \label{vf}%
\end{equation}

\begin{lemma}
\label{lem:vf}The vector field $\chi_{\tau}$ has the following properties:

\begin{enumerate}
\item[(a)] \emph{div}$\chi_{\tau}=N-k,$

\item[(b)] $\left\langle \mathrm{d}\chi_{\tau}(y,z)\left[  \xi\right]
,\xi\right\rangle \leq\max\{1-k\varphi(\left\vert y\right\vert ),1\}\left\vert
\xi\right\vert ^{2}$ \ for every $y\in\mathbb{R}^{k+1}\smallsetminus\{0\},$
$z\in\mathbb{R}^{N-k-1},$ $\xi\in\mathbb{R}^{N}.$
\end{enumerate}
\end{lemma}

\begin{proof}
See \cite[Lemma 2.3]{pa2} or \cite[Lemma 4.2]{cfp}.
\end{proof}

\begin{proposition}
\label{prop:nonexistence}Assume there exists $\tau\in(0,\infty)$ such that
$\left\vert y\right\vert \in(\tau,\infty)$ for every $(y,z)\in\Omega$ and
$\left\langle \chi_{\tau},\nu_{\Omega}\right\rangle >0$ a.e. on $\partial
\Omega.$ If $p\geq2_{N,k}^{\ast}$ and%
\[
\lambda\leq\frac{2(p-2_{N,k}^{\ast})}{2_{N,k}^{\ast}(p-2)}\lambda_{1}(\Omega),
\]
then problem \emph{(\ref{prob_p})} does not have a nontrivial solution.
\end{proposition}

\begin{proof}
The variational identity (4) in Pucci and Serrin's paper \cite{ps} implies
that, if $u\in\mathcal{C}^{2}(\Omega)\cap\mathcal{C}^{1}(\overline{\Omega})$
is a solution of (\ref{prob_p}) and $\chi\in\mathcal{C}^{1}(\overline{\Omega
},\mathbb{R}^{N}),$ then%
\begin{align}
\frac{1}{2}\int_{\partial\Omega}\left\vert \nabla u\right\vert ^{2}%
\left\langle \chi,\nu_{\Omega}\right\rangle d\sigma= &  \int_{\Omega}\left(
\text{div}\chi\right)  \left[  \frac{1}{p}\left\vert u\right\vert ^{p}%
+\frac{\lambda}{2}u^{2}-\frac{1}{2}\left\vert \nabla u\right\vert ^{2}\right]
dx\label{pu-se}\\
&  +\int_{\Omega}\left\langle \mathrm{d}\chi\left[  \nabla u\right]  ,\nabla
u\right\rangle dx,\nonumber
\end{align}
where $\nu_{\Omega}$ is the outward pointing unit normal to $\partial\Omega$
(in the notation of \cite{ps} we have taken $\mathcal{F}(x,u,\nabla
u)=\frac{1}{2}|\nabla u|^{2}-\frac{1}{2}\lambda u^{2}-\frac{1}{p}|u|^{p}$,
$h=\chi$ and $a=0$). Let $\chi:=\chi_{\tau}.$ Then, by Lemma \ref{lem:vf},
\[
\text{div}\chi_{\tau}=N-k.
\]
Moreover, since $1-k\varphi(t)<1$ for $t\in(\tau,\infty),$ and $\left\vert
y\right\vert \in(\tau,\infty)$ for every $(y,z)\in\Omega$, Lemma
\ref{lem:vf}\ yields
\[
\left\langle \mathrm{d}\chi_{\tau}(y,z)\left[  \xi\right]  ,\xi\right\rangle
\leq\left\vert \xi\right\vert ^{2}\qquad\forall(y,z)\in\Omega,\text{ }\xi
\in\mathbb{R}^{N}.
\]
By assumption, $\left\langle \chi_{\tau},\nu_{\Omega}\right\rangle >0$ a.e. on
$\partial\Omega.$ Therefore, if $u$ is a nontrivial solution of (\ref{prob_p})
we have, using \eqref{pu-se}, that%
\begin{align*}
0 &  <(N-k)\left(  \frac{1}{p}-\frac{1}{2}\right)  \int_{\Omega}\left[
\left\vert \nabla u\right\vert ^{2}-\lambda u^{2}\right]  dx+\int_{\Omega
}\left\vert \nabla u\right\vert ^{2}dx\\
&  =(N-k)\left(  \frac{1}{p}-\frac{1}{2}+\frac{1}{N-k}\right)  \int_{\Omega
}\left\vert \nabla u\right\vert ^{2}dx-(N-k)\left(  \frac{1}{p}-\frac{1}%
{2}\right)  \lambda\int_{\Omega}u^{2}dx,
\end{align*}
that is,%
\[
\left(  \frac{1}{2}-\frac{1}{p}\right)  \lambda\int_{\Omega}u^{2}dx>\left(
\frac{1}{2_{N,k}^{\ast}}-\frac{1}{p}\right)  \int_{\Omega}\left\vert \nabla
u\right\vert ^{2}dx.
\]
Therefore, if $p\geq2_{N,k}^{\ast}$ and%
\[
\lambda\leq\frac{2(p-2_{N,k}^{\ast})}{2_{N,k}^{\ast}(p-2)}\inf_{\substack{u\in
H_{0}^{1}(\Omega)\\u\neq0}}\frac{\int_{\Omega}\left\vert \nabla u\right\vert
^{2}dx}{\int_{\Omega}u^{2}dx}=\frac{2(p-2_{N,k}^{\ast})}{2_{N,k}^{\ast}%
(p-2)}\lambda_{1}(\Omega),
\]
problem (\ref{prob_p}) does not have a nontrivial solution in $\Omega,$ as claimed.
\end{proof}

The following result was proved in \cite{cfp}.

\begin{proposition}
\label{prop:**}If $\Theta$ is doubly starshaped then $\Theta\subset\left(
t_{0},\infty\right)  \times\mathbb{R}^{N-k-1}$ and $\left\langle \chi_{t_{0}%
},\nu_{\Theta}\right\rangle >0$ a.e. on $\partial\Theta,$ with $t_{0}$ as in
\emph{Definition \ref{def:**}}.
\end{proposition}

\begin{proof}
See the proof of (4.11) in \cite{cfp}.
\end{proof}

\begin{proof}
[Proof of Theorem \ref{thm:nonexistence}.]The conclusion follows immediately
from Propositions \ref{prop:nonexistence} and \ref{prop:**}.
\end{proof}

\section{Existence and nonexistence of symmetric ground states to
supercritical problems}

\label{sec:proof}Next, we come back to our original supercritical problem%
\begin{equation}
\left\{
\begin{array}
[c]{ll}%
-\Delta v=\lambda v+\left\vert v\right\vert ^{2_{N,k}^{\ast}-2}v & \text{in
}\Omega,\\
\hspace{0.6cm}v=0 & \text{on }\partial\Omega,
\end{array}
\right.  \tag{$\wp_\lambda$}%
\end{equation}
where
\[
\Omega:=\{(y,z)\in\mathbb{R}^{k+1}\times\mathbb{R}^{N-k-1}:\left(  \left\vert
y\right\vert ,z\right)  \in\Theta\}
\]
for some bounded smooth domain $\Theta$ in $\mathbb{R}^{N-k}$ with
$\overline{\Theta}\subset\left(  0,\infty\right)  \times\mathbb{R}^{N-k-1}$,
$1\leq k\leq N-3$, and $2_{N,k}^{\ast}:=\frac{2(N-k)}{N-k-2}$.

An $O(k+1)$-invariant function $v:\Omega\rightarrow\mathbb{R}$ can be written
as $v(y,z)=u(\left\vert y\right\vert ,z)$ for some function $u:\Theta
\rightarrow\mathbb{R}$. A straightforward computation shows that%
\begin{equation}
\Delta v=\frac{1}{a(x)}\text{div}(a(x)\nabla u)\text{,} \label{laplacian}%
\end{equation}
where $a(x_{1},\ldots,x_{N-k}):=x_{1}^{k}.$ Hence, $v$ is an $O(k+1)$%
-invariant solution of (\ref{prob}) if and only if $u$ solves%
\begin{equation}
\left\{
\begin{array}
[c]{ll}%
-\text{div}(x_{1}^{k}\nabla u)=\lambda x_{1}^{k}u+x_{1}^{k}\left\vert
u\right\vert ^{2^{\ast}{-2}}u & \text{in }\Theta,\\
\hspace{1.7cm}u=0 & \text{on }\partial\Theta,
\end{array}
\right.  \tag{$\wp_\lambda^\#$}\label{prob2}%
\end{equation}
where $2^{\ast}=2_{N,k}^{\ast}$ is the critical exponent in dimension
$n:=N-k=$ dim$(\Theta).$ So this problem is a special case of the problem
treated in section \ref{sec:anisotropic} with $a(x_{1},\ldots,x_{n}%
):=x_{1}^{k}$ and $a=b=c$.

For these functions $a,b,c$ we simplify notation and write $\ell_{\lambda
}^{\left[  k\right]  },$ $\kappa^{\left[  k\right]  },$ $\lambda_{m}^{\left[
k\right]  },$ $\lambda_{m,\ast}^{\left[  k\right]  }$ instead of
$\ell_{\lambda}^{a,b,c},$ $\kappa^{a,c},$ $\lambda_{m}^{a,b},$ $\lambda
_{m,\ast}^{a,b,c}$. Note that%
\[
\kappa^{\left[  k\right]  }=\left(  \min\limits_{x\in\overline{\Theta}}%
x_{1}^{k}\right)  \frac{1}{n}S^{n/2}.
\]

\begin{proposition}
\label{prop:nogroundstate}If $\alpha:=\min_{x\in\overline{\Theta}}x_{1}$ and
$\lambda\leq0,$ then $\ell_{\lambda}^{\left[  k\right]  }=\frac{\alpha^{k}}%
{n}S^{n/2}$ and it is not attained by $J_{\lambda}$ on $\mathcal{N}_{\lambda
}\equiv\mathcal{N}_{\lambda}(\Theta).$
\end{proposition}

\begin{proof}
By Theorem \ref{thm:bifurcation} it is enough to show this for $\lambda=0.$
Arguing by contradiction, assume that $\ell_{0}^{\left[  k\right]  }%
<\frac{\alpha^{k}}{n}S^{n/2}.$ Then there exists $\varphi\in\mathcal{C}%
_{c}^{\infty}(\Theta)\cap\mathcal{N}_{0}(\Theta)$\ such that
\[
J_{0}(\varphi)<\frac{\alpha^{k}}{n}S^{n/2}=\kappa^{\left[  k\right]  }.
\]
Since supp$(\varphi)$ is a compact subset of $\left(  \alpha,\infty\right)
\times\mathbb{R}^{n-1}$, there exists a $\varrho\in\left(  \alpha
,\infty\right)  $ such that supp$(\varphi)\subset B:=\left\{  x\in
\mathbb{R}^{n}:(x_{1}-\varrho)^{2}+x_{2}^{2}+\cdots+x_{n}^{2}<(\alpha
-\varrho)^{2}\right\}  .$ Hence, $\varphi\in\mathcal{N}_{0}(B).$ Theorem
\ref{thm:nonexistence} and the discussion given at the beginning of this
section imply that problem%
\[
-\text{div}(x_{1}^{k}\nabla u)=x_{1}^{k}\left\vert u\right\vert ^{{2}^{\ast
}-2}u\quad\text{in}\ B,\qquad u=0\quad\text{on}\ \partial B
\]
does not have a nontrivial solution. So, by Theorem \ref{thm:bifurcation},
$\inf_{\mathcal{N}_{0}(B)}J_{0}=\kappa^{\left[  k\right]  }=\frac{\alpha^{k}%
}{n}S^{n/2}.$ But%
\[
\inf_{\mathcal{N}_{0}(B)}J_{0}\leq J_{0}(\varphi)<\frac{\alpha^{k}}{n}%
S^{n/2}.
\]
This is a contradiction. We conclude that $\ell_{0}^{\left[  k\right]  }%
=\frac{\alpha^{k}}{n}S^{n/2}.$

Since this value is the same for every $\Theta$ such that $\alpha:=\min
_{x\in\overline{\Theta}}x_{1},$ a standard argument shows that $\ell
_{0}^{\left[  k\right]  }$ is not attained by $J_{0}$ on $\mathcal{N}%
_{0}\equiv\mathcal{N}_{0}(\Theta).$
\end{proof}

Set%
\[
\alpha:=\min\limits_{x\in\overline{\Theta}}x_{1}\text{,\qquad}\beta
:=\max\limits_{x\in\overline{\Theta}}x_{1}.
\]

\begin{lemma}
\label{lem:f}For every positive function $f\in C^{2}[\alpha,\beta]$ which
satisfies%
\begin{equation}
\alpha^{k}f^{2}(\alpha)\leq t^{k}f^{2}(t)\quad\text{and}\quad t^{k}f^{2^{\ast
}}(t)\leq\alpha^{k}f^{2^{\ast}}(\alpha)\qquad\forall t\in\lbrack\alpha,\beta],
\label{f}%
\end{equation}
we have that%
\[
\lambda_{0,\ast}^{\left[  k\right]  }\geq\min_{t\in\lbrack\alpha,\beta]}%
\frac{-\left(  t^{k}f^{\prime}(t)\right)  ^{\prime}}{t^{k}f(t)}.
\]

\end{lemma}

\begin{proof}
Let $u\in H_{0}^{1}(\Theta),$ $u\neq0,$ and set $u(x)=f(x_{1})w(x)$. Then%
\begin{align*}
\int_{\Theta}x_{1}^{k}\left\vert \nabla u\right\vert ^{2}  &  =\int_{\Theta
}\left(  x_{1}^{k}f^{2}\left\vert \nabla w\right\vert ^{2}+x_{1}^{k}f^{\prime
}f\frac{\partial w^{2}}{\partial x_{1}}+x_{1}^{k}(f^{\prime})^{2}w^{2}\right)
\\
&  =\int_{\Theta}\left(  x_{1}^{k}f^{2}\left\vert \nabla w\right\vert
^{2}+x_{1}^{k}f^{\prime}\frac{\partial(fw^{2})}{\partial x_{1}}\right) \\
&  =\int_{\Theta}\left(  x_{1}^{k}f^{2}\left\vert \nabla w\right\vert
^{2}-\left(  x_{1}^{k}f^{\prime}\right)  ^{\prime}fw^{2}\right)  .
\end{align*}
So, if $\lambda\leq\frac{-\left(  t^{k}f^{\prime}(t)\right)  ^{\prime}}%
{t^{k}f(t)}$ for all $t\in\lbrack\alpha,\beta]$, we have that%
\begin{align*}
&  \frac{\int_{\Theta}x_{1}^{k}\left\vert \nabla u\right\vert ^{2}-\lambda
\int_{\Theta}x_{1}^{k}u^{2}}{\left(  \int_{\Theta}x_{1}^{k}\left\vert
u\right\vert ^{2^{\ast}}\right)  ^{2/2^{\ast}}}=\frac{\int_{\Theta}x_{1}%
^{k}f^{2}\left\vert \nabla w\right\vert ^{2}-\int_{\Theta}\left[  \left(
x_{1}^{k}f^{\prime}\right)  ^{\prime}+\lambda x_{1}^{k}f\right]  fw^{2}%
}{\left(  \int_{\Theta}x_{1}^{k}f^{2^{\ast}}\left\vert w\right\vert ^{2^{\ast
}}\right)  ^{2/2^{\ast}}}\\
&  \qquad\geq\frac{\int_{\Theta}x_{1}^{k}f^{2}\left\vert \nabla w\right\vert
^{2}}{\left(  \int_{\Theta}x_{1}^{k}f^{2^{\ast}}\left\vert w\right\vert
^{2^{\ast}}\right)  ^{2/2^{\ast}}}\geq\frac{\alpha^{k}f^{2}(\alpha
)\int_{\Theta}\left\vert \nabla w\right\vert ^{2}}{\alpha^{2k/2^{\ast}}%
f^{2}(\alpha)\left(  \int_{\Theta}\left\vert w\right\vert ^{2^{\ast}}\right)
^{2/2^{\ast}}}\\
&  \qquad=\alpha^{2k/n}\frac{\int_{\Theta}\left\vert \nabla w\right\vert ^{2}%
}{\left(  \int_{\Theta}\left\vert w\right\vert ^{2^{\ast}}\right)
^{2/2^{\ast}}}\geq\alpha^{2k/n}S>0\qquad\text{for all }u\in H_{0}^{1}%
(\Theta),u\neq0.
\end{align*}
This implies that $\lambda<\lambda_{1}^{\left[  k\right]  }$ and, hence, that%
\[
\max_{t>0}J_{\lambda}(tu)=\frac{1}{n}\left(  \frac{\int_{\Theta}x_{1}%
^{k}\left\vert \nabla u\right\vert ^{2}-\lambda\int_{\Theta}x_{1}^{k}u^{2}%
}{\left(  \int_{\Theta}x_{1}^{k}\left\vert u\right\vert ^{2^{\ast}}\right)
^{2/2^{\ast}}}\right)  ^{n/2}\geq\frac{\alpha^{k}}{n}S^{n/2}%
\]
for all $u\in H_{0}^{1}(\Theta),u\neq0.$ Therefore, $\ell_{\lambda}^{\left[
k\right]  }=\frac{\alpha^{k}}{n}S^{n/2}$ for every $\lambda\leq\min
_{t\in\lbrack\alpha,\beta]}\frac{-\left(  t^{k}f^{\prime}(t)\right)  ^{\prime
}}{t^{k}f(t)},$ and the conclusion follows.
\end{proof}

We obtain the following estimates for $\lambda_{0,\ast}^{\left[  k\right]  }.$

\begin{proposition}
\label{prop:k>1}$\lambda_{0,\ast}^{\left[  k\right]  }\geq0$ and
\[
\lambda_{0,\ast}^{\left[  k\right]  }\geq\left\{
\begin{array}
[c]{ll}%
\frac{(k-1)^{2}}{4\beta^{2}} & \text{if }2k\geq n,\\
\frac{k}{\left(  2^{\ast}\beta\right)  ^{2}}\left(  (2^{\ast}-1)k-2^{\ast
}\right)  & \text{if }2k\leq n.
\end{array}
\right.
\]
Therefore $\lambda_{0,\ast}^{\left[  k\right]  }>0$ if $k\geq2.$
\end{proposition}

\begin{proof}
Proposition \ref{prop:nogroundstate} implies that $\lambda_{0,\ast}^{\left[
k\right]  }\geq0.$

Set $f(t):=t^{-\gamma}$ with $\frac{k}{2^{\ast}}\leq\gamma\leq\frac{k}{2}.$
This function satisfies (\ref{f}) and, since
\[
\frac{-\left(  t^{k}f^{\prime}(t)\right)  ^{\prime}}{t^{k}f(t)}=\frac
{\gamma(k-\gamma-1)}{t^{2}},
\]
Lemma \ref{lem:f}\ implies that%
\[
\lambda_{0,\ast}^{\left[  k\right]  }\geq\frac{\gamma(k-\gamma-1)}{\beta^{2}%
}.
\]
Now observe that the the function $\phi(\gamma):=\gamma(k-\gamma-1)$\ attains
its maximum on the interval $\left[  \frac{k}{2^{\ast}},\frac{k}{2}\right]  $
at the point%
\[
\gamma_{\ast}:=\max\left\{  \frac{k-1}{2},\frac{k}{2^{\ast}}\right\}  .
\]
Therefore $\lambda_{0,\ast}^{\left[  k\right]  }\geq\frac{1}{\beta^{2}}%
\phi(\gamma_{\ast})=\frac{1}{\beta^{2}}\gamma_{\ast}(k-\gamma_{\ast}-1),$ as claimed.

Finally, note that $k>\frac{2^{\ast}}{2^{\ast}-1}=\frac{2n}{n+2}$ if $k\geq2$.
Hence, $\lambda_{0,\ast}^{\left[  k\right]  }>0$ if $k\geq2.$\smallskip
\end{proof}

\begin{proof}
[Proof of Theorem \ref{thm:main}.]Using Corollary \ref{cor:ps} it is easily
seen that, if $v(y,z)=u(\left\vert y\right\vert ,z),$ then $v$ is an
$O(k+1)$-invariant ground state for problem (\ref{prob}) if and only if $u$ is
a ground state for problem (\ref{prob2})$.$ So Theorem \ref{thm:main} follows
immediately from Proposition \ref{prop:nogroundstate}, Theorem
\ref{thm:bifurcation} and Proposition \ref{prop:k>1}.
\end{proof}

The following result shows that $\lambda_{0,\ast}^{\left[  1\right]  }>0$ if
the domain is thin enough in the $x_{1}$-direction.

\begin{proposition}
\label{prop:k=1}If $\frac{\beta}{\alpha}\leq\frac{n}{n-2}$ then $\lambda
_{0,\ast}^{\left[  k\right]  }\geq\frac{k^{2}}{4\beta^{2}}>0$ for all
$k\geq1.$
\end{proposition}

\begin{proof}
Set $f(t):=e^{-\gamma(t-\alpha)}$ with $\frac{k}{2^{\ast}\alpha}\leq\gamma
\leq\frac{k}{2\beta},$ and write$\ g(t):=t^{k}f^{2}(t)\ $and$\ h(t):=t^{k}%
f^{2^{\ast}}(t).$ Then $g^{\prime}(t)=t^{k-1}e^{-2\gamma(t-\alpha)}(k-2\gamma
t)\geq0$ and $h^{\prime}(t)=t^{k-1}e^{-2^{\ast}\gamma(t-\alpha)}(k-2^{\ast
}\gamma t)\leq0$ for all $t\in\lbrack\alpha,\beta],$ so $f$ satisfies
(\ref{f}). Since%
\[
\frac{-\left(  t^{k}f^{\prime}(t)\right)  ^{\prime}}{t^{k}f(t)}=\frac{\gamma
t^{k-1}e^{-\gamma(t-\alpha)}(k-\gamma t)}{t^{k}e^{-\gamma(t-\alpha)}}%
=\frac{\gamma(k-\gamma t)}{t},
\]
Lemma \ref{lem:f}\ implies that%
\[
\lambda_{0,\ast}^{\left[  k\right]  }\geq\frac{\gamma(k-\gamma\beta)}{\beta}.
\]
Now observe that the the function $\phi(\gamma):=\gamma(k-\gamma\beta
)$\ attains its maximum at the point $\gamma_{\ast}:=\frac{k}{2\beta}.$ Hence,
$\lambda_{0,\ast}^{\left[  k\right]  }\geq\frac{k^{2}}{4\beta^{2}}>0,$ as claimed.
\end{proof}

\section{Some open questions and comments}

Many questions remain open. Here are some of them.

\begin{problem}
Concerning problem $($\ref{prob2}$)$:

\begin{enumerate}
\item Is it true that $\lambda_{0,\ast}^{\left[  1\right]  }>0$ for any domain
$\Theta,$ and not only for thin domains?

\item For $m\geq1,$ is $\lambda_{m,\ast}^{\left[  k\right]  }>\lambda
_{m}^{\left[  k\right]  },$ or is $\lambda_{m,\ast}^{\left[  k\right]
}=\lambda_{m}^{\left[  k\right]  }$ as in the classical Brezis-Nirenberg case?

\item What happens in general at $\lambda_{m,\ast}^{\left[  k\right]  }?$ Is
there, or not, a ground state of problem $($\ref{prob2}$)$ for $\lambda
=\lambda_{m,\ast}^{\left[  k\right]  }?$
\end{enumerate}
\end{problem}

\begin{problem}
Concerning the general anisotropic problem \emph{(\ref{ancrit})}:

\begin{enumerate}
\item Is $\lambda_{0,\ast}^{a,b,c}$ always nonnegative? Or are there examples
where a ground state exists for some $\lambda<0$? For all $\lambda<0?$

\item Can one give lower estimates for $\lambda_{m,\ast}^{a,b,c}$ in some cases?

\item Suppose that $c\in\mathcal{C}^{1}(\overline{\Theta})$ in addition to our
earlier assumptions. If $\kappa^{a,c}$ is attained only at points which are
non-stationary for $\frac{a(x)^{n/2}}{c(x)^{(n-2)/2}}$ and lie on the boundary
of $\Theta$, is it then true that $\lambda_{0,\ast}^{a,b,c}>0$?
\end{enumerate}
\end{problem}

Two particular cases of (3) are: $c=1,$ and $a=b=c$. If the answer is positive
in the first case, this would be in contrast to the results in \cite{e} and
\cite{hy}. A positive answer in the second case would be a generalization of
our results for \eqref{prob2}. A partial answer can be given using Proposition
\ref{prop:k>1}. Consider, for example, the problem%
\begin{equation}
-\text{div}(a(x)\nabla u)=\lambda b(x)u+\left\vert u\right\vert ^{{2}^{\ast
}-2}u\quad\text{in}\ \Theta,\qquad u=0\quad\text{on}\ \partial\Theta,
\label{prob4}%
\end{equation}
where $\Theta$ is a bounded smooth domain in $\mathbb{R}^{n},$ $n\geq3,$
$\lambda\in\mathbb{R}$, $a\in\mathcal{C}^{1}(\overline{\Theta}),$
$b\in\mathcal{C}^{0}(\overline{\Theta})$ are strictly positive on
$\overline{\Theta},$ and $2^{\ast}=\frac{2n}{n-2}.$ Then, the following
statement holds true.

\begin{proposition}
If $a(x)\geq x_{1}^{k}\geq b(x)$ for all $x\in\Theta$ and $\min_{x\in
\overline{\Theta}}a(x)=(\min_{x\in\overline{\Theta}}x_{1})^{k}>0$ for some
$k\geq2$, then $\lambda_{0,\ast}^{a,b,1}>0.$
\end{proposition}

\begin{proof}
Let $\alpha:=\min_{x\in\overline{\Theta}}x_{1}>0.$ For every $u\in H_{0}%
^{1}(\Theta),$ $u\neq0,$ $\lambda\in\lbrack0,\lambda_{0,\ast}^{[k]}]$ we have
that
\[
\frac{\int_{\Theta}a(x)\left\vert \nabla u\right\vert ^{2}-\lambda\int
_{\Theta}b(x)u^{2}}{\alpha^{2k/2^{\ast}}\left(  \int_{\Theta}\left\vert
u\right\vert ^{2^{\ast}}\right)  ^{2/2^{\ast}}}\geq\frac{\int_{\Theta}%
x_{1}^{k}\left\vert \nabla u\right\vert ^{2}-\lambda\int_{\Theta}x_{1}%
^{k}u^{2}}{\left(  \int_{\Theta}x_{1}^{k}\left\vert u\right\vert ^{2^{\ast}%
}\right)  ^{2/2^{\ast}}}>0.
\]
Hence, $\lambda<\lambda_{1}^{a,b}$ and
\begin{align*}
\frac{1}{n}\left(  \frac{\int_{\Theta}a(x)\left\vert \nabla u\right\vert
^{2}-\lambda\int_{\Theta}b(x)u^{2}}{\left(  \int_{\Theta}\left\vert
u\right\vert ^{2^{\ast}}\right)  ^{2/2^{\ast}}}\right)  ^{n/2}  &  \geq
\alpha^{\frac{k(n-2)}{2}}\frac{1}{n}\left(  \frac{\int_{\Theta}x_{1}%
^{k}\left\vert \nabla u\right\vert ^{2}-\lambda\int_{\Theta}x_{1}^{k}u^{2}%
}{\left(  \int_{\Theta}x_{1}^{k}\left\vert u\right\vert ^{2^{\ast}}\right)
^{2/2^{\ast}}}\right)  ^{n/2}\\
&  \geq\alpha^{\frac{k(n-2)}{2}}\frac{\alpha^{k}}{n}S^{n/2}=\left(  \alpha
^{k}\right)  ^{n/2}\frac{1}{n}S^{n/2}=\kappa^{a,1}.
\end{align*}
It follows from Theorem \ref{thm:bifurcation} that
\[
\ell_{\lambda}^{a,b,1}=\kappa^{a,1}\text{\qquad for all }\lambda\in
(-\infty,\lambda_{0,\ast}^{[k]}].
\]
Hence, by Proposition \ref{prop:k>1}, $\lambda_{0,\ast}^{a,b,1}\geq
\lambda_{0,\ast}^{[k]}>0,$ as claimed.
\end{proof}

\end{document}